\pdfoutput=1
\documentclass[12pt]{amsart}
\usepackage[leqno]{amsmath}
\usepackage{listings,amssymb,graphicx,url,tikz,ctable,setspace,mathtools,float,makecell}
\usetikzlibrary{arrows.meta, positioning}
\usepackage{enumerate}
\usepackage[symbol]{footmisc}

\input{glyphtounicode}
\allowdisplaybreaks

\usepackage[colorlinks = true,
            linkcolor = blue,
            urlcolor  = red,
            citecolor = green,
            anchorcolor = blue]{hyperref}

\usepackage{perpage}
\MakePerPage{footnote}

\theoremstyle{definition}
\newtheorem*{Definition}{Definition}

\theoremstyle{plain}
\newtheorem{Theorem}{Theorem}[section]
\newtheorem{Lemma}[Theorem]{Lemma}
\newtheorem{Proposition}[Theorem]{Proposition}
\newtheorem{Corollary}[Theorem]{Corollary}

\theoremstyle{remark}
\newtheorem*{Remark}{Remark}

\newcommand{\C}{\mathbb{C}}
\newcommand{\Z}{\mathbb{Z}}
\newcommand{\N}{\mathbb{N}}
\newcommand{\Q}{\mathbb{Q}}

\DeclarePairedDelimiter{\abs}{\lvert}{\rvert}
\newcommand{\Frac}{\operatorname{Frac}}
\newcommand{\EDS}{\operatorname{EDS}}
\newcommand{\ES}{\operatorname{ES}}

\AtBeginDocument{%
  \mathchardef\mathcomma\mathcode`\,
  \mathcode`\,="8000 
}
{\catcode`,=\active
  \gdef,{\mathcomma\discretionary{}{}{}}
}

\calclayout

\title{On elliptic sequences over commutative rings}
\author{Junyan Xu}
\address{Heidelberg University, Germany.}
\email{junyanxu.math@gmail.com}
\begin{document}

\maketitle
\begin{abstract}
We define elliptic sequences over a commutative ring as sequences indexed by the (positive) integers satisfying a 4-parameter, highly symmetric family of homogeneous quartic relations among terms which we call elliptic relations.
We classify elliptic sequences over a field into three types, and show that most of them are dilated multiples of standard elliptic divisibility sequences (EDSs) which form countably many 4-dimensional families.
In particular, we show standard EDSs are elliptic in a purely algebraic way using intricate implications among elliptic relations, without relying on complex analytic theory of Weierstrass functions.
We shall use results presented here to give a purely algebraic treatment of division polynomials in a follow-up paper.
\end{abstract}

\section{Introduction} \label{intro}

For the purpose of this paper, an \textbf{elliptic sequence} is a sequence $h = (h_n)_{n>0}$ indexed by the positive integers $\Z^+$ with terms in a commutative ring $R$ satisfying the \textbf{elliptic relations}
\begin{equation}
    E(a,b,c,d): h_{a+b}h_{a-b}h_{c+d}h_{c-d}=h_{a+c}h_{a-c}h_{b+d}h_{b-d}-h_{b+c}h_{b-c}h_{a+d}h_{a-d} \label{ell_rel4} 
\end{equation}
for all integers $a>b>c>d\ge0$. These sequences are closely connected to elliptic curves (hence the name) and have been studied (mainly for $R=\Z$) ever since Ward's 1948 \textit{Memoir on Elliptic Divisibility Sequences} \cite{Ward}. It is also common to consider $\Z$-indexed elliptic sequences, for which $E(a,b,c,d)$ is required to hold for arbitrary $a,b,c,d\in\Z$, and every $\Z^+$-indexed elliptic sequence can be extended to a $\Z$-indexed one by setting $h_0=0$ and $h_n=-h_{-n}$ for $n<0$. The relations $E(a,b,c,d)$ are easily seen to be equivalent to the axiom for \textbf{elliptic nets} defined by Stange (formula (2) in \cite{Stange})\footnote{\cite{Stange} considers more generally elliptic nets indexed by an arbitrary free finitely-generated abelian group that is not necessarily $\Z$.}
\begin{equation}
  EN(p,q,r,s): h_{p + q + s} h_{p - q} h_{r + s} h_r
+ h_{q + r + s} h_{q - r} h_{p + s} h_p
+ h_{r + p + s} h_{r - p} h_{q + s} h_q = 0 \label{ell_net}
\end{equation}
via the transformation $(a,b,c,d) = \left(p+\frac{s}2, q+\frac{s}2, r+\frac{s}2, \frac{s}2\right)$\footnote{Since $s$ can be odd, it can happen that $a,b,c,d$ are all half-integers instead of integers, but in such cases we can still find integral $a',b',c',d'$ such that $E(a,b,c,d)$ is equivalent to $E(a',b',c',d')$, see the second paragraph of \S\ref{implication}}, or $(p,q,r,s)=(a-d,b-d,c-d,2d)$, assuming $h_{r-p}=-h_{p-r}$. The equivalence has been recognized in \cite{StangeFormulary}, \S1.2 and is attributed to Nelson Stephens there.
The identity $E(a,b,c,d)$ with $h$ replaced by the Weierstrass $\sigma$ function is called the ``three-term equation" and attributed to Weierstrass in \cite{WhittakerWatson}, Example 20.5.6, which holds for complex variables, not just integral indices.

The 4-parameter family $E(a,b,c,d)$ of elliptic relations is highly overdetermined, 
because a 1-parameter family is already sufficient to specify a sequence by recursion. Indeed, prior work used smaller families to define elliptic sequences, with slight variations in conventions, terminology, and generality.
The 2-parameter family
\begin{equation*}
E(m,n,1,0): h_{m+n}h_{m-n}h_1^2 = h_{m+1}h_{m-1}h_n^2 - h_{n+1}h_{n-1}h_m^2 
\end{equation*}
is used to define elliptic sequences in \cite{Ward}
and \cite{Swart}\footnote{In fact both papers uses inhomogeneous relations with the $h_1^2$ factor omitted, but as remarked in VII.23 there, elliptic sequences (according to his definition) with $h_1^2\ne1$  are uninteresting. \cite{Ward} and \cite{Swart} considered sequences of integers and rational numbers respectively.
\cite{Ward} considered sequences indexed by nonnegative integers $\N$ and $m\ge n\ge0$, which makes little difference (we can always set $h_0=0$), while \cite{Swart} considered $\Z$-indexed sequences.}
and elliptic divisibility sequences (EDSs) over a field in \cite{Silverman}, Exercise 3.34. The 3-parameter family
\begin{equation*}
E(m,n,r,0): h_{m+n}h_{m-n}h_{r}^{2} = h_{m+r}h_{m-r}h_{n}^{2} - h_{n+r}h_{n-r}h_{m}^{2} 
\end{equation*}
is used to define  elliptic sequences in Lean's Mathlib which is called ``generalised elliptic sequence" in \cite{Swart}, \S4.1.2\footnote{Both \cite{Swart} and \href{https://leanprover-community.github.io/mathlib4_docs/Mathlib/NumberTheory/EllipticDivisibilitySequence.html\#IsEllSequence}{Mathlib} use $\Z$-indexed sequences, with $m,n,r\in\Z$ arbitrary.}.
If $h_1$ is not a zero-divisor in $R$, the 2-parameter family implies the 3-parameter family,
because of the identity
\begin{equation}
h_1^2 E(m,n,r,0)=h_r^2 E(m,n,1,0)-h_n^2 E(m,r,1,0)+h_m^2 E(n,r,1,0). \label{ell_rel3_rel}
\end{equation}

Among all possible 1-parameter families of elliptic relations, two will be our subjects of study, namely the \textbf{even-odd recurrence}
\begin{align*}
E(n+1,n,1,0):h_{2n+1}h_1^2 &=h_{n+2}h_n^3-h_{n-1}h_{n+1}^3 & \text{for }n\ge 2 
\\
E(n+1,n-1,1,0):h_{2n}h_2h_1^2 &=h_n(h_{n+2}h_{n-1}^2-h_{n+1}^2h_{n-2}) & \text{for }n \ge 3 
\end{align*}
and the \textbf{Somos 4 recurrence}
\begin{align*}
     \qquad\qquad E(n,2,1,0):\quad h_{n+2}h_{n-2}h_1^2&=h_2^2h_{n+1}h_{n-1}-h_3h_1h_n^2 & \text{for }n\ge 3. 
\end{align*}
If $h_2 h_1$ is not a zero-divisor, the even-odd recurrence determines the whole elliptic sequence from four initial terms $h_1,h_2,h_3,h_4$, and the same is true for the Somos 4 recurrence if no $h_n$ is a zero-divisor. 

\begin{Remark}
The even-odd recurrence is traditionally used to define the division polynomials $\psi_n$ associated to an elliptic curve, a classical example of an EDS.
$E(a,b,c,d)$ is homogeneous of degree 4 and weighted homogeneous of degree $2(a^2+b^2+c^2+d^2)$ if $h_n$ is assigned weight $n^2$, and therefore also weighted homogeneous if $h_n$ is assigned weight $C_1 n^2+C_2$ for arbitrary constants $C_1$ and $C_2$. This is consistent with the division polynomial $\psi_n$ having weighted degree $n^2-1$. (Each $\psi_n$ is a polynomial in 7 variables $X,Y,a_1,a_2,a_3,a_4,a_6$ with weights $w(X)=2$, $w(Y)=3$ and $w(a_i)=i$.)
\end{Remark}

Our first major results in \S\ref{implication} shows that sequences defined by either of these two recurrences are elliptic, that is, either of the 1-parameter family implies the full 4-parameter family, provided the non-zero-divisor conditions.
These can be deduced for $R=\C$ by combining \cite{Ward}, Theorem 12.1\footnote{This assumes a 2-parameter family of relations, but the proof can be adapted to assume a 1-parameter family only.} with the aforementioned three-term equation for the Weierstrass $\sigma$ function, but
we shall present an elementary, purely algebraic proof valid for an arbitrary commutative ring $R$, which uses relations among elliptic relations like (\ref{ell_rel3_rel}) together with some transformation rules connecting equivalent elliptic relations.
Poorten and Swart \cite{PoortenSwart} was able to prove that the Somos 4 recurrence implies all $E(m,n,1,0)$ using a roundabout, albeit elementary argument, 
and \cite{ShipseySwart} claims before Theorem 6 that the method can be adapted to prove that the even-odd recurrence also implies all $E(m,n,1,0)$. These can then be combined with Corollary \ref{ell_rel_from_one_zero} to further derive all $E(a,b,c,d)$,
but our approach does not need this intermediate step and appears cleaner and clarifying. 

In \S\ref{std_EDS}, we show that given arbitrary elements $r_2,r_3,r_4\in R$, we can define an elliptic sequence $(h_n)_{n>0}$ such that $h_1=1$, $h_2=r_2$, $h_3=r_3$, $h_4=r_2 r_4$ and every $h_n$ is a polynomial in $r_2,r_3,r_4$. Moreover, $(h_n)_{n>0}$ is a \textbf{divisibility sequence}, meaning that $h_n\mid h_m$ whenever $n\mid m$. 
We denote it by $\EDS(r_2,r_3,r_4)$ and call it a standard elliptic divisibility sequence (EDS).
The definition of standard EDSs is implicit in 
the definition of division polynomials 
(especially \cite{Lercier}, Définition 8) 
and divisibility was already treated by Ward (\cite{Ward}, Theorem 4.1),
but we offer a streamlined
treatment valid over any commutative ring.

By homogeneity of the elliptic relations, we can further multiply every term of a standard EDS by an arbitrary $r_1\in R$ to obtain a 4-dimensional family of EDSs. Combined with the identity principle in \S\ref{id_principle} 
that says an elliptic sequence is determined by 4 initial terms if $h_2 h_1$ is not a zero-divisor, this suggests that a major irreducible component of the algebraic set of elliptic sequences is rational and 4-dimensional if $R$ is a field. Indeed, we completely classify elliptic sequences over a field in \S\ref{classification} and show they come in rational families of dimensions 4, 3, and 2 in a self-similar pattern and are all closely connected to standard EDSs. We also consider sequences satisfying 1-, 2-, and 3-parameter families of elliptic relations and show that the 3-parameter family implies being elliptic over a field. We end the paper with some initial investigation into the scheme of elliptic sequences, defined to be the spectrum of the polynomial ring with variables indexed by $\Z^+$ modulo all elliptic relations.

In \S\ref{translation_invariant}, we obtain a slight generalization to a ``translation invariant" found by Swart in \cite{Swart} and studied in some subsequent work. 

The work was motivated by the pursuit of an elementary, purely algebraic proof of the formula for the $n$th multiple of a point on an elliptic curve in terms of division polynomials (see \cite{Silverman}, Exercise 3.7) for formalization in Lean, which has been successful (joint work with David Kurniadi Angdinata) and will be the subject of a follow-up paper. 
Most results\footnote{The exceptions are Corollary \ref{ell_rel_from_one_zero}, Theorem \ref{ell_rel_from_Somos}, Proposition \ref{Somos_ext} and \ref{even_odd_ext}, Theorem \ref{neg_ext}, and \S\ref{classification}.} in this paper are included in the pull request to Lean's Mathlib, in the file $\textsf{EllipticDivisibilitySequence.lean}$
\footnote{\url{https://github.com/leanprover-community/mathlib4/pull/13782/changes\#diff-02d4668be95cd41b094c2ec5859acf1cf92c74833b9181a533bcce3a23c00961}}.

\section{Implications among Elliptic Relations} \label{implication}

Each of the three terms of an elliptic relation have four factors which are terms of the sequence $h$. We focus on the left-hand side of $E(a,b,c,d)$, whose four factors have indices $a\pm b,c\pm d$ respectively. There are three ways to partition the four factors into two pairs. In one of the partitions, the pair $a+b,c+d$ can be written as $a'\pm b'$ and the pair $a-b,c-d$ as $c'\pm d'$, with $(a',b',c',d'):=\left(\frac{a+b+c+d}2,\frac{a+b-c-d}2,\frac{a+c-b-d}2,\left|\frac{a+d-b-c}{2}\right|\right)$. Furthermore, this transformation also induces a different partition of the factors in both terms on the right-hand side of $E(a,b,c,d)$, leaving the product invariant, which shows that $E(a,b,c,d)$ is equivalent to $E(a',b',c',d')$, and we call this transformation \textbf{Rule I.1}. Similarly, from the $a+b,c-d\mid a-b,c+d$ partition, we see that $E(a,b,c,d)$ is equivalent to $E\left(\frac{a+b+c-d}2,\frac{a+b+d-c}2,\frac{a+c+d-b}2,\left|\frac{b+c+d-a}2\right|\right)$, which we call \textbf{Rule I.2}.

Notice that if there is an even number (0, 2, or 4) of odd numbers among $a,b,c,d$, then the indices remain integers after either transformation, but if there is an odd number (1 or 3) of odd numbers, then the indices become half-integers. Conversely, if $a,b,c,d$ are all half-integers, then $a\pm b$ have opposite parity, and so do $c\pm d$, so we can partition the four indices into a pair of even numbers $a'\pm b'$ and a pair of odd numbers $c'\pm d'$, which forces $a',b',c',d'\in\Z$. Therefore, every elliptic relation with half-integer indices is equivalent to one with integer indices, so an elliptic sequence also satisfies $E(a,b,c,d)$ for all half-integers $a,b,c,d$ (with the $a>b>c>d\ge0$ restriction for an $\Z^+$-indexed sequence). 

The absolute value sign in the last transformed parameter could be omitted if we consider $\Z$-indexed sequences $h$ satisfying $h_{-n}=-h_n$ for all $n\in\Z$. Under this condition, negating any number of the four parameters $a,b,c,d$ leaves the three terms of $E(a,b,c,d)$ invariant, and permuting the parameters permutes the terms up to sign, resulting in an equivalent relation in both cases, so an arbitrary elliptic relation $E(a,b,c,d)$ is equivalent to one with $a\ge b\ge c\ge d\ge 0$. If moreover $h_0=0$, then whenever two of the four parameters coincide, one of the terms in $E(a,b,c,d)$ is 0 while the other two terms cancel, so $E(a,b,c,d)$ trivially holds. Therefore, if $(h_n)_{n\in\Z}$ satisfies $h_0=0$, $h_{-n}=-h_n$, and $E(a,b,c,d)$ for all $a>b>c>d\ge0$, then it actually satisfies $E(a,b,c,d)$ for all $a,b,c,d\in\Z$ (and all $a,b,c,d\in\Z+\frac12$), i.e. it is elliptic. If we start with an $\Z^+$-indexed elliptic sequence, we can extend it to a $\Z$-indexed elliptic sequence by setting $h_0=0$ and $h_n=-h_{-n}$ for $n<0$.

\begin{Remark}
The lattice in $\mathbb{R}^4$ consisting of quadruples of integers and of half-integers is the dual $D_4^*$ of the $D_4$ lattice of quadruples of integers summing up to an even number; it is similar to $D_4$ itself via e.g. the map $(a,b,c,d)\mapsto (a+b,a-b,c+d,c-d)$ that scales the Euclidean distance by $\sqrt{2}$. The two I rules (ignoring the absolute value signs) 
are reflections along the root vectors $(\frac12,-\frac12,-\frac12,-\frac12)$ and $(\frac12,-\frac12,-\frac12,\frac12)$ in $D_4^*$ respectively, so they are elements of the Weyl group $W(D_4^*)$, but not $W(D_4)$. Both of the root systems $D_4$ and $D_4^*$ have 24 roots, and their disjoint union $F_4$ has 48. Both of the Weyl groups $W(D_4)$ and $W(D_4^*)$ have order 192 and are contained in $W(F_4)$ (order $6\times 192=1152$) which is the full symmetry group of $D_4$ or $D_4^*$ fixing the origin. See \cite{Brown} for more about connections between elliptic nets and $F_4$.
\end{Remark}

Having exhausted implications (equivalences) between two elliptic relations, we will next see how an elliptic relation could follow from multiple elliptic relations.
\textbf{Rule II.1} derives two relations from three: let $m,n,r,c,d$ be arbitrary (distinct) nonnegative integers or half-integers such that $E(m,n,c,d)$, $E(m,r,c,d)$ and $E(n,r,c,d)$ all hold, and assume $h_{c+d}h_{c-d}$ is not a zero divisor, then $E(m,n,r,c)$ and $E(m,n,r,d)$ also hold. The proof is by observing
$$h_{c+d}h_{c-d}E(m,n,r,c)=h_{r+c}h_{r-c}E(m,n,c,d)-h_{n+c}h_{n-c}E(m,r,c,d)+h_{m+c}h_{m-c}E(n,r,c,d)$$
and a similar identity with $E(m,n,r,d)$ in the left-hand side, of which (\ref{ell_rel3_rel}) is a special case. If we introduce the abbreviation
$$T(mn|abcd) := h_{m+n}h_{m-n}E(a,b,c,d),$$
these can be written as
\begin{align*}
T(cd|mnrc)&=T(rc|mncd)-T(nc|mrcd)+T(mc|nrcd) \\
T(cd|mnrd)&=T(rd|mncd)-T(nd|mrcd)+T(md|nrcd).
\end{align*}

\textbf{Rule II.2} also assumes $h_{c+d}h_{c-d}$ is not a zero divisor and derives one relation from ten by observing that
\begin{align*}
T(cd|mnrs)&=T(nd|mrsc)-T(rd|mnsc)+T(sd|mnrc)\\
&+T(nc|mrsd)-T(rc|mnsd)+T(sc|mnrd)\\
&+T(nr|mscd)-T(ns|mrcd)+T(rs|mncd)\\
&-2T(md|nrsc).
\end{align*}

These implications are all we need to derive all elliptic relations from a 1-parameter family.

\begin{Lemma}
\label{ell_rel_min_case}
Let $a_0\ge0$ be an integer or a half-integer, and define
\begin{align*}
    c_{\min}=1,\ d_{\min}=0&\text{ if }a_0\text{ is an integer, and}\\
    c_{\min}=\frac32,\ d_{\min}=\frac12&\text{ if }a_0\text{ is a half-integer}.
\end{align*}
If $h=(h_n)_{n>0}$ satisfies $E(a,b,c_{\min},d_{\min})$ whenever $a_0\ge a>b>c_{\min}$, and $h_{c_{\min}+d_{\min}}h_{c_{\min}-d_{\min}}$ is not a zero-divisor, then $h$ satisfies $E(a,b,c,d)$ whenever $a_0\ge a>b>c>d\ge 0$.
\end{Lemma}
\begin{proof}
For arbitrary $a,b,c$ satisfying $a_0\ge a>b>c>c_{\min}$, we derive $E(a,b,c,c_{\min})$ and $E(a,b,c,d_{\min})$ from $E(a,b,c_{\min},d_{\min})$, $E(a,c,c_{\min},d_{\min})$ and $E(b,c,c_{\min},d_{\min})$ using Rule II.1. For arbitrary $a,b,c,d$ satisfying $a_0\ge a>b>c>d>c_{\min}$, we derive $E(a,b,c,d)$ from $E(a,c,d,c_{\min})$, $E(a,b,d,c_{\min})$, $E(a,b,c,c_{\min})$, $E(a,c,d,d_{\min})$, $E(a,b,d,d_{\min})$, $E(a,b,c,d_{\min})$, $E(a,d,c_{\min},d_{\min})$, $E(a,c,c_{\min},d_{\min})$, $E(a,b,c_{\min},d_{\min})$, and $E(b,c,d,c_{\min})$ using Rule II.2. 
Both applications of Rule II requires the assumption that $h_{c_{\min}+d_{\min}}h_{c_{\min}-d_{\min}}$ is not a zero-divisor. 
Since the above two cases together with the 
assumption exhaust all possibilities of $(a,b,c,d)$ satisfying $a_0\ge a>b>c>d\ge 0$, the theorem is proved.
\end{proof}

\begin{Corollary}
\label{ell_rel_from_one_zero}
    If $h=(h_n)_{n>0}$ satisfies $E(m,n,1,0)$ for all integers $m>n>1$, and $h_1$ is not a zero-divisor, then $h$ is elliptic.
\end{Corollary}
\begin{proof}
We need to prove $E(a,b,c,d)$ for integers $a>b>c>d\ge0$, for which we have $c_{\min}=1$ and $d_{\min}=0$, so $h_{c_{\min}+d_{\min}}h_{c_{\min}-d_{\min}}=h_1^2$, which is not a zero-divisor since we assume that $h_1$ is not. Since we assume that $h$ satisfies $E(a,b,1,0)=E(a,b,c_{\min},d_{\min})$ for all $a>b>1$, applying Theorem \ref{ell_rel_min_case} we see that $h$ also satisfies $E(a,b,c,d)$.
\end{proof}

\begin{Theorem}
\label{ell_rel_from_even_odd}
If $h=(h_n)_{n>0}$ satisfies the even-odd recurrence, and $h_2h_1$ is not a zero-divisor, then $h$ is elliptic.
\end{Theorem}
\begin{proof}\footnote{This proof was first posted as \href{https://math.stackexchange.com/a/4903422/12932}{an answer on StackExchange}, using the same preparatory lemmas and implication rules.}
We prove that $\forall b\ \forall c\ \forall d\ (a>b>c>d\ge0 \to E(a,b,c,d))$ by induction on $a$, where $a,b,c,d$ are either all integers or all half-integers. By Theorem \ref{ell_rel_min_case} we only need to prove that $h$ satisfies $E(a,b,c_{\min},d_{\min})$, since $h_{c_{\min}+d_{\min}}h_{c_{\min}-d_{\min}}$, which is equal to $h_1^2$ in the integer case and $h_2h_1$ in the half-integer case, is not a zero-divisor by our assumption. In both cases we have $\frac{a+b+c-d}2=\frac{a+b+1}2\le a$ since $a > b$ implies $a\ge b+1$. If $a>b+1$, we apply Rule I.2 to transform $E(a,b,c,d)$ to $E\left(a',b',c',d'\right)$ which has a smaller first argument $a'=\frac{a+b+c-d}2<a$ and therefore is true by induction hypothesis. If $a=b+1$, in the integer case we have to prove $E(b+1,b,1,0)$ which is the odd recurrence, and in the half-integer case we apply Rule I.1 to show that $E(b+1,b,\frac32,\frac12)$ is equivalent to $E(b+\frac32,b-\frac12,1,0)$ which is the even recurrence.
\end{proof}

Allowing half-integers is crucial to this proof, because we cannot derive e.g. $E(6,2,1,0)$ from the even-odd recurrence by applying the rules only to elliptic relations with integer parameters.

Using similar techniques, we can also prove that the Somos 4 recurrence also implies all elliptic relations:
\begin{Theorem} \label{ell_rel_from_Somos}
    If $(h_n)_{n>0}$ satisfies $E(n,2,1,0)$ for $n\le a+b-2$, and $h_n$ is not a zero-divisor for any $n \le a+b-4$, then $h$ satisfies $E(a,b,c,d)$.
\end{Theorem}
\begin{proof}
    By induction, we may assume that $E(a',b',c',d')$ has been proven for all $a'>b'>c'>d'\ge0$ such that either
\begin{enumerate}[(i)]
    \item $a'+b'<a+b$, or 
    \item $a'+b'=a+b$ and $a'<a$, or
    \item $a'+b'=a+b$ and $a'=a$ and $d'>d$.
\end{enumerate}
If $c > d+1$, we use a variant of Rule II.1
\[ T(ce|abcd)=T(cd|abce)+T(bc|aced)-T(ac|bced) \]
to derive $E(a,b,c,d)$ from $E(a,b,c,d+1)$, $E(a,c,d+1,d)$ and $E(b,c,d+1,d)$ which satisfy (iii), (i) and (i) respectively; $h_{c+(d+1)}h_{c-(d+1)}$ is not a zero-divisor because $c+(d+1)\le c+c-1\le (a-2)+(b-1)-1=a+b-4$. If $c=d+1$ but $b < a-1$, Rule I.2 transforms $E(a,b,c,d)$ to some $E(a',b',c',d')$ satisfying (ii). If $c=d+1$, $b=a-1$ and $c < b-1$, we use another variant of Rule II.1 (just swap $c$ and $e$ in the previous formula)
\[ T(ec|abcd)=T(cd|abec)+T(bc|aecd)-T(ac|becd) \]
to derive $E(a,b,c,d)$ from $E(a,b,c+1,c)$, $E(a,c+1,c,d)$ and $E(b,c+1,c,d)$ which satisfy (iii), (i) and (i) respectively; $h_{(c+1)+c}h_{(c+1)-c}$ is not a zero-divisor because $(c+1)+c=2(c+1)-1\le 2(b-1)-1=b+b-3\le (a-1)+b-3=a+b-4$. Finally, if $c=d+1$, $b=a-1$ and $c=b-1$, then $(a,b,c,d)=(d+3,d+2,d+1,d)$ and Rule I.1 transforms $E(a,b,c,d)$ to $E(2d+3,2,1,0)$, and we just need to verify $2d+3\le a+b-2$: in fact equality holds.
\end{proof}
We shall provide a simple alternative proof after we define standard EDSs in \S\ref{std_EDS} that uses Theorem \ref{ell_rel_from_even_odd} together with an identity principle to be introduced next.

\section{Identity Principles} \label{id_principle}

In this section we prove several identity principles which state that sequences satisfying certain families of elliptic relations are determined by certain initial terms.

\begin{Proposition} \label{Somos_ext}
    If $(h_n)_{n>0}$ and $(h'_n)_{n>0}$ are two sequences satisfying $E(n,2,1,0)$ for all $n\le m-2$ such that $h_n=h'_n$ for $1\le n\le 4$ and $h_n$ is not a zero-divisor for any $n\le m-4$, then $h_n=h'_n$ for $n\le m$.
\end{Proposition}
\begin{proof}
By induction on $n$. We may assume that $n>4$. The left-hand side of $E(n-2,2,1,0)$ for a sequence $h$ is $h_n h_{n-4}h_1^2$ and right-hand side involve terms with indices $<n$, which by induction hypothesis are the same for $h$ and $h'$. By assumption $h_{n-4}h_1^2=h'_{n-4}{h'_1}^2$ is not a zero-divisor, from which we conclude $h_n=h'_n$.
\end{proof}
The next theorem will be useful to classify elliptic sequences $h$ over a field with $h_2 h_1=0$.
\begin{Proposition} \label{even_odd_ext}
    If $r,s\in\Z^+$ are of opposite parity and $r < s$, and 
    $(h_n)_{n>0}$ and $(h'_n)_{n>0}$ are two sequences satisfying $E(2r+i,r+i,r,0)$ and $E(s+r+i,r+i,r,0)$ for all $i\in\Z^+$
    such that $h_n=h'_n$ for all $n\le 3r$ of the same parity as $r$ and all $n \le s+2r$ of the same parity as $s$ (which holds if $h_n=h'_n$ for all $n\le s+2r$ irrespective of parity), and $h_sh_r$ is not a zero-divisor, then $h=h'$.
\end{Proposition}
\begin{proof}
    We show by induction that $h_n=h'_n$ for all $n$; by induction, assume $h_m=h'_m$ for all $m<n$. If $n$ has the same parity as $r$, then $h_n=h'_n$ by assumption if $n\le 3r$, so assume $n>3r$ and write $n=3r+2i$.
    By $E(2r+i,r+i,r,0):h_{3r+2i}h_r^3=h_{3r+i}h_{r+i}^3-h_{2r+i}^3h_i$, we see that $h_{3r+2i}$ is determined by terms with smaller indices since $h_r$ is not a zero-divisor. Since $h$ and $h'$ both satisfy $E(2r+i,r+i,r,0)$ and agree on indices smaller than $n=3r+2i$, we conclude that $h_n=h_{3r+2i}=h'_{3r+2i}=h'_n$ as well. If $n$ has the same parity as $s$, we write $n=s+2r+2i$ and use $E(s+r+i,r+i,r,0):h_{s+2r+2i}h_sh_r^2=h_{s+2r+i}h_{s+i}h_{r+i}^2-h_{2r+i}h_{s+r+i}^2h_i$ instead.
\end{proof}
Taking $r=1$ and $s=2$ we immediately get
\begin{Corollary} \label{even_odd_ext_cor}
    If $(h_n)_{n>0}$ and $(h'_n)_{n>0}$ are two sequences satisfying the even-odd recurrence such that $h_n=h'_n$ for $1\le n\le 4$ and $h_2h_1$ is not a zero-divisor, then $h=h'$.
\end{Corollary}

Finally we point out some elliptic relations that ensures $h_0=0$ and/or $h_{-n}=-h_n$ for a $\Z$-indexed sequence $h$. Firstly, $E(0,0,0,0)$ says $h_0^4=0$, which implies $h_0=0$ if $R$ is reduced (e.g. is a domain, or a field). More generally, $E(2n,n,n,0)$ says $h_{2n}^3 h_0=0$, so if moreover $h_{2n}$ is not a zero-divisor, then $h_0=0$. On the other hand, $E(a,b,c,0)+E(b,a,c,0)$ says $(h_{a-b}+h_{-(a-b)})h_{a+b}h_c^2=0$. If at least one odd term $h_c$ is not a zero-divisor, then $a-b=n$ and $a+b=c$ is solvable for odd $n\in\Z$, and we conclude that $h_{-n}=h_{-(a-b)}=-h_{a-b}=-h_n$ for all odd $n$. Similarly, if at least one even term is not a zero-divisor, then $h_{-n}=-h_n$ for all even $n$. If $R$ is a domain, then either all odd terms are zero, or there is a nonzero odd term, which is not a zero-divisor: in either case $h_{-n}=-h_n$ holds for all odd $n$, and similarly for all even $n$. In summary, we have

\begin{Theorem} \label{neg_ext}
    A sequence $(h_n)_{n\in\Z}$ in an integral domain satisfying $E(m,n,r,0)$ for all $m,n,r\in\Z$ also satisfy $h_0=0$ and $h_{-n}=-h_n$ for all $n\in\Z$.
\end{Theorem}

\section{Standard Elliptic Divisibility Sequences}
\label{std_EDS}

As we saw in the previous section, both the even-odd recurrence and the Somos 4 recurrence determine a sequence subject to non-zero-divisor conditions. The even-odd recurrence only requires that $h_1$ and $h_2$ are not zero-divisors, and is what we will use to define standard elliptic divisibility sequences (EDSs)\footnote{\href{https://leanprover-community.github.io/mathlib4_docs/Mathlib/NumberTheory/EllipticDivisibilitySequence.html\#normEDS}{\textsf{normEDS}} in Mathlib.}. First we define $h_1=1$, then the even-odd recurrence becomes
\begin{align*}
h_{2n+1} &=h_{n+2}h_n^3-h_{n-1}h_{n+1}^3 & \text{for }n\ge 2, \\
h_{2n}h_2 &=h_n(h_{n+2}h_{n-1}^2-h_{n+1}^2h_{n-2}) & \text{for }n \ge 3.
\end{align*}
In order for the the second equation (the even recurrence) to have a solution, $h_2$ must divide $h_n(h_{n+2}h_{n-1}^2-h_{n+1}^2h_{n-2})$.
Since $h$ is supposed to be a divisibility sequence, we should have $h_2\mid h_m$ whenever $m$ is even; since every term in $h_n(h_{n+2}h_{n-1}^2-h_{n+1}^2h_{n-2})$ contains two factors with even subscripts and two with odd subscripts, if we know that $h_2\mid h_{2m}$ for $m < n$ then $h_2^2\mid h_n(h_{n+2}h_{n-1}^2-h_{n+1}^2h_{n-2})=h_2h_{2n}$, so not only a solution $h_{2n}$ exists, it can be chosen to be divisible by $h_2$.
From the base cases $h_2\mid h_2$ and $h_2\mid h_4$, this argument inductively shows that all $h_n$ are well-defined and $h_2\mid h_{2n}$, and inspires us to first define an auxiliary sequence\footnote{\href{https://leanprover-community.github.io/mathlib4_docs/Mathlib/NumberTheory/EllipticDivisibilitySequence.html\#preNormEDS}{\textsf{preNormEDS}} in Mathlib.} $\tilde{h}$ which is $h$ with a $h_2$ factor removed from the even terms. Indeed we can define $(\tilde{h}_n)_{n\in\Z}$ from initial values $\tilde{h}_3, \tilde{h}_4$ and an additional parameter $h_2\in R$ by the even-odd recurrence
\begin{align*}
    \tilde{h}_0 &= 0, \quad \tilde{h}_1=\tilde{h}_2=1 \\
    \tilde{h}_{2n+1} &= h_2^4\tilde{h}_{n+2}\tilde{h}_n^3-\tilde{h}_{n-1}\tilde{h}_{n+1}^3 &&\text{for }n\ge 2\text{ even} \\
    \tilde{h}_{2n+1} &= \tilde{h}_{n+2}\tilde{h}_n^3-h_2^4\tilde{h}_{n-1}\tilde{h}_{n+1}^3 &&\text{for }n\ge 3\text{ odd} \\
    \tilde{h}_{2n} &= \tilde{h}_n(\tilde{h}_{n+2}\tilde{h}_{n-1}^2-\tilde{h}_{n+1}^2\tilde{h}_{n-2}) &&\text{for }n\ge 3 \\
    \tilde{h}_n&=-\tilde{h}_{-n}&&\text{for }n<0.
\end{align*}
\begin{Definition} \label{std_EDS_def}
  The sequence $(h_n)_{n\in\Z}$ defined by $h_n=\tilde{h}_n$ for $n$ odd and $h_n=h_2\tilde{h}_n$ for $n$ even is called a standard EDS and denoted $\EDS(h_2,\tilde{h}_3,\tilde{h}_4)$.
\end{Definition}
This idea was already present in Lercier's doctoral thesis \cite{Lercier} (1997), which comes two years after the joint paper \cite{LercierMorain} and adjusted the definition of division polynomials $\psi_n$ (Définition 8 in \S2.1.2) to use $f_n$, which is our $\tilde{h}_n$ but with $h_2^4$ replaced by $F^2=(h_2^2-4W)^2$, 
so his $\psi_n$ only satisfies the elliptic relations modulo the Weierstrass polynomial $W$. The idea was independently rediscovered by Angdinata during his formalization of standard EDSs in Lean's Mathlib and communicated to the author.

By induction we see that all $\tilde{h}_n$, and therefore all terms of a standard EDS, are polynomials with integer coefficients in the parameters $h_2,\tilde{h}_3$ and $\tilde{h}_4$, and there is no need to require $h_2$ to be a non-zero-divisor. It is clear from the definition that a standard EDS satisfies the even-odd recurrence, which implies it is elliptic by Theorem \ref{ell_rel_from_even_odd} if $h_2$ is not a zero-divisor. This still holds if $h_2$ is a zero-divisor, but we have to resort to a ``reduction to the universal case" argument: every $h=\EDS(h_2,\tilde{h}_3,\tilde{h}_4)$ is a specialization of the \textbf{universal} (standard) \textbf{EDS} $h^U=\EDS(X_2,X_3,X_4)$ over $\Z[X_2,X_3,X_4]$, in the sense that $h^U_n$ evaluated at $(h_2,\tilde{h}_3,\tilde{h}_4)$ gives $h_n$. Since $h^U_2=X_2\ne 0$ is not a zero-divisor in the domain $\Z[X_2,X_3,X_4]$, $h^U$ is elliptic, so its specialization $h$ is also elliptic.
We can in fact show $h^U_n\ne0$ for $n\ne 0$ because $h^U$ specializes to $h=\EDS(2,3,2)$ which is the identity sequence $h_n=n$ by the identity principle Corollary \ref{even_odd_ext_cor}.

Next we show that every $\EDS(h_2,\tilde{h}_3,\tilde{h}_4)$ is a divisibility sequence, i.e. for all $m,n\in\Z$ we have $h_m\mid h_{nm}$. 
It suffices to prove this for $h=h^U$, since a witness for $h^U_m\mid h^U_{nm}$ specializes (evaluates) to a witness for $h_m\mid h_{nm}$. 
We first assume $n\ge0$ and proceed by induction on $n$. The $n=0$ case is trivial because $h_0=0$ and the $n=1$ case reduces to the trivial divisibility $h_m\mid h_m$. For $n=2$ we have
$h_{2n}=h_2\tilde{h}_{2n}=h_2\tilde{h}_{n}(\tilde{h}_{n-1}^2\tilde{h}_{n+2}-\tilde{h}_{n-2}\tilde{h}_{n+1}^2)$ and clearly $h_n\mid h_2\tilde{h}_n$ no matter $n$ is even or odd, so $h_n\mid h_{2n}$. If $n\ge 3$ is even, by induction we may assume $h_m\mid h_{(n/2)m}$ and by the $n=2$ case we conclude that $h_m\mid h_{(n/2)m}\mid h_{2(n/2)m}=h_{nm}$. If $n=2k+1\ge 3$ is odd and $m\ne0$, we make use the elliptic relation
\begin{equation*}
    E((k+1)m,km,1,0): h_{(2k+1)m}h_mh_1^2=h_{(k+1)m+1}h_{(k+1)m-1}h_{km}^2-h_{km+1}h_{km-1}h_{(k+1)m}^2,
\end{equation*}
which becomes
\begin{equation*}
  \frac{h_{(2k+1)m}}{h_m}=\frac{h_{(k+1)m+1}}{h_1}\frac{h_{(k+1)m-1}}{h_1}\left(\frac{h_{km}}{h_m}\right)^2-\frac{h_{km+1}}{h_1}\frac{h_{km-1}}{h_1}\left(\frac{h_{(k+1)m}}{h_m}\right)^2
\end{equation*}
after dividing both sides by $h_m^2h_1^2$. This is possible since $\Z[X_2,X_3,X_4]$ is a domain and $h^U_m\ne0$. Since $k\le k+1<2k+1$, by induction we have $h_m\mid h_{km}$ and $h_m\mid h_{(k+1)m}$, so the right-hand side lies in $\Z[X_2,X_3,X_4]$ 
and witnesses the divisibility $h_m\mid h_{(2k+1)m}$. 
(Notice that this proof recursively produces divisibility witnesses without using polynomial division.) We conclude that
\begin{Theorem}
\label{std_EDS_is_EDS}
    Every $\EDS(h_2,\tilde{h}_3,\tilde{h}_4)$ is both elliptic and a divisibility sequence, i.e. an EDS.
\end{Theorem}

A corollary of divisibility of standard EDSs is that they also satisfy $h_3h_2\mid h_{n+1}h_nh_{n-1}$. First notice that $h^U_3=X_3$ and $h^U_2=X_2$ both divide $h^U_6$ and $X_3$ is coprime to $X_2$, so it must be the case that $h^U_3h^U_2\mid h^U_6$ (or one may compute $h^U_6=X_2 X_3(X_2^4X_4-X_3^3-X_4^2)$), implying $h_3h_2\mid h_6$ for all standard EDSs $h$. This shows $h_3h_2\mid h_{n+1}h_nh_{n-1}$ if $n\equiv0,\pm1\pmod 6$. For the remaining cases, if $n\equiv \pm2\pmod 6$ then $h_3\mid h_{n\pm1}$ and $h_2\mid h_n$, while if $n\equiv 3\pmod 6$ then $h_3\mid h_n$ and $h_2$ divides both $h_{n+1}$ and $h_{n-1}$.

\begin{Theorem} \label{EDS_standard}
    If $(h_n)_{n>0}$ satisfies the even-odd recurrence, $h_2h_1$ is not a zero divisor, $h_1\mid h_2\mid h_4$ and $h_1\mid h_3$, then $h=h_1\EDS(h_2/h_1,h_3/h_1,h_4/h_2)$.
\end{Theorem}
\begin{proof}
Both $h$ and $h':=h_1\EDS(h_2/h_1,h_3/h_1,h_4/h_2)$ satisfy the even-odd recurrence, and $h_2 h_1$ is not a zero-divisor, so by the identity principle Corollary \ref{even_odd_ext_cor}, it suffices to check that $h_n=h'_n$ for $1\le n\le 4$, which is trivial.
\end{proof}

\begin{Remark}
If $R$ is a field, then any nonzero element is not a zero-divisor and divides any element, so every elliptic sequence over $R$ with $h_2h_1\ne0$ is $h_1$ times a standard EDS. If $R$ is a domain, then it embeds into a field, so elliptic sequences $h$ over $R$ with $h_2h_1\ne0$ are also reasonably well understood, though it is still nontrivial to determine when an elliptic sequence in $\Frac(R)$ lies in $R$: for example, $\EDS(2,4,3/2)$ lies in $\Z$ (it is easiest to prove inductively that $2\mid h_n$ for $n\ge 5$) but $h_2=2\nmid h_4=3$, so it is only a divisibility sequence over $\Q$, not $\Z$; $\EDS(2,3,3/2)$ starts with $1,2,3,3,-3,-63/2$ so the first four terms but not all are in $\Z$.
\end{Remark}

We can now give another proof of Theorem \ref{ell_rel_from_Somos}:
\begin{proof}
If $a+b=5$ then $E(a,b,c,d)=E(3,2,1,0)$ and there is nothing to prove, so we assume $a+b\ge 6$, so that $a+b-4\ge 2$. By assumption, $h_1$ and $h_2$ are not zero-divisors. Consider the sequence $h':=h_1\EDS(h_2/h_1,h_3/h_1,h_4/h_2)$ over $\Frac(R)$\footnote{Here $\Frac(R)$ denotes the total ring of fractions, i.e. the localization of $R$ with respect to the submonoid of non-zero-divisors. Its key property is that the canonical ring homomorphism $R\to\Frac(R)$ is injective.}, which satisfies all elliptic relations $E(a,b,c,d)$ by Theorem \ref{std_EDS_is_EDS}, in particular all $E(n,2,1,0)$. By assumption $h_n$ is not a zero-divisor in $R$, and therefore not a zero-divisor in $\Frac(R)$, for $n\le a+b-4$, so we may apply the identity principle Proposition \ref{Somos_ext} to conclude that $h_n=h'_n$ for $n \le a+b$. Therefore, $h$ satisfies $E(a,b,c,d)$, because $h'$ satisfies it and the largest subscript in $E(a,b,c,d)$ is $a+b$.
\end{proof}

\section{Translation Invariant}
\label{translation_invariant}

An elliptic sequence $h$ with $h_1=1$ satisfies the Somos 4 recurrence
\begin{equation}
h_{n+2}h_{n-2}=Ah_{n+1}h_{n-1}-Bh_n^2 \label{Somos_general}
\end{equation}
with $A=h_2^2$ and $B=h_3$, and Swart observed in \cite{Swart} (see \S7.2.1, Remark 2 on p. 161--162) that $(h_{n+2}h_{n-1}^2+h_{n+1}^2h_{n-2}+Ah_n^3)/h_{n+1}h_nh_{n-1}$ is a \textbf{translation invariant}\footnote{See the Remarks following Theorem 1.2 
in \cite{HoneSwart} for literature review and the beginning of \S5 for some discussions.} independent of $n$. Indeed, multiplying a relation of the form
\begin{equation}
    C(h_{n+2}h_{n-1}^2+h_{n+1}^2h_{n-2}+Ah_n^3)=Dh_{n+1}h_nh_{n-1} \label{Somos_invariant}
\end{equation}
by $h_{n+2}$, we obtain
\begin{align*}
   Dh_{n+2}h_{n+1}h_nh_{n-1}&=C(h_{n+2}^2h_{n-1}^2+h_{n+2}h_{n-2}h_{n+1}^2+Ah_{n+2}h_n^3) \\
   &= C(h_{n+2}^2h_{n-1}^2+(Ah_{n+1}h_{n-1}-Bh_n^2)h_{n+1}^2+(h_{n+3}h_{n-1}+Bh_{n+1}^2)h_n^2) \\
   &= C(h_{n+3}h_n^2+h_{n+2}^2h_{n-1}+Ah_{n+1}^3)h_{n-1},
\end{align*}
where the second identity uses the Somos 4 recurrence (\ref{Somos_general}) twice, at $n$ and $n+1$. If $h_{n-1}$ is not a zero-divisor, we can cancel it from both sides and arrive at the same relation (\ref{Somos_invariant}) but with $n$ replaced by $n+1$. In general, if $n < m$ and $h_{n-1},h_n,\dots,h_{m-2}$ are not zero-divisors, it follows by induction
$$ (h_{n+2}h_{n-1}^2+h_{n+1}^2h_{n-2}+Ah_n^3)h_{m+1}h_mh_{m-1}=h_{n+1}h_nh_{n-1}(h_{m+2}h_{m-1}^2+h_{m+1}^2h_{m-2}+Ah_m^3). $$
Specializing to the universal EDS $h=h^U$ and $n=2$, and cancelling $h_3h_2$ from both sides, we obtain
\begin{equation*}
  h_{m+2}h_{m-1}^2+h_{m+1}^2h_{m-2}+h_2^2h_m^3 = (\tilde{h}_4+h_2^4)h_{m+1}h_mh_{m-1}/h_3,
\end{equation*}
which therefore still holds if $h$ is an arbitrary standard EDS.\footnote{Notice that $h^U_{m+1}h^U_m h^U_{m-1}/h_3^U$ is a well-defined polynomial in $\Z[X_2,X_3,X_4]$, since 3 divides one of the indices $m+1$, $m$, and $m-1$ and $h^U$ is a divisibility sequence. Evaluating it gives a definition of $h_{m+1}h_m h_{m-1}$ for arbitrary standard EDS $h$.}
This will be used in a follow-up paper to derive an expression for the polynomials $\omega_n$ associated to an elliptic curve (see \cite{Silverman}, Exercise 3.7) that continue to work in characteristic 2.

For an elliptic sequence $(h_n)_{n\in\Z}$, we can in fact prove the more general formula ${N(n,s)D(m,s)} = D(n,s)N(m,s)$ (which morally says that $N(n,s)/D(n,s)$ is a translation invariant) for all $m,n,s\in\Z$, where $N(n,s):=h_s^2(h_{n+2s}h_{n-s}^2+h_{n+s}^2 h_{n-2s})+h_{2s}^2 h_n^3$ and $D(n,s)=h_{n+s}h_n h_{n-s}$ using the linear combination
\begin{align*}
& h_m h_n h_{2s}^2 EN(m,n,s,0) \\
& - h_s^2(h_{m-s} h_{n-s} EN(m,n,s,s) + h_{m+s}h_{n+s} EN(m-s,n-s,s,s) - h_{m-n} h_{2s} EN(n+s,n,n-s,m-n)),
\end{align*}
and this does not require any term to be a non-zero-divisor.

\section{Classification of elliptic sequences over a field}
\label{classification}

The main goal of this last section of this paper
is to completely classify $\Z^+$-indexed sequences $h$ over a field satisfying all elliptic relations $E(a,b,c,d)$, but we also discuss sequences that only satisfy $E(m,n,r,0)$, $E(m,n,1,0)$, the even-odd recurrence or the Somos 4 recurrence, obtaining partial results. We will consider the following cases separately: (0) $h_1=0$; (100) $h_1\ne0$, $h_2=h_3=0$; (101) $h_1\ne 0$, $h_2=0$, $h_3\ne0$; (11) $h_1\ne 0$, $h_2\ne0$. Results are summarized in the following table:

\begin{table}[H]
\begin{tabular}{|c|cccc|}
\hline
    & \multicolumn{1}{c|}{Somos 4 rec.}                                & \multicolumn{1}{c|}{even-odd rec.}                    & \multicolumn{1}{c|}{$E(m,n,1,0)$} & \makecell{$E(m,n,r,0)$\\$E(a,b,c,d)$}         \\ \hline
0   & \multicolumn{1}{c|}{$h_2=0$ or $(\forall n) h_n h_{n+2}=0$} & \multicolumn{2}{c|}{$h_n h_{n+2}=0$}                                                 & Theorem \ref{classification_thm} \\ \hline
100 & \multicolumn{1}{c|}{$h_n h_{n+4}=0$}                      & \multicolumn{1}{c|}{TBD ($\infty$-dim. families)} & \multicolumn{2}{c|}{$A\cdot\ES_{1,\ell}(B,D)$}                                \\ \hline
101 & \multicolumn{4}{c|}{$A\cdot\EDS(0,C,0)$}                                                                                                                                                                   \\ \hline
11  & \multicolumn{1}{c|}{\S6.1--2}              & \multicolumn{3}{c|}{$A\cdot\EDS(B,C,D)$}                                                                                                       \\ \hline
\end{tabular}
\end{table}

In case (11), if $h$ satisfies the even-odd recurrence, then by Theorem \ref{EDS_standard} that $h=h_1\EDS(h_2/h_1,h_3/h_1,h_4/h_2)$ and therefore satisfies all $E(a,b,c,d)$. If $h$ satisfies the Somos 4 recurrence instead, we need to split into cases $h_3=0$ and $h_3\ne 0$, and we will discuss each in a subsections below.

In case (0), the Somos 4 recurrence $E(n,2,1,0)$ reduces to $h_{n+1}h_{n-1}h_2^2=0$, so either $h_2=0$ and there is no other constraints, or $h_2\ne 0$ and $h_n$ for $n\ge 3$ can be arbitrary as long as no two nonzero terms are exactly two indices apart from each other. For the even-odd recurrence, $E(n+1,n,1,0)$ reduces to $h_{n+2}h_n^3=h_{n+1}^3h_{n-1}$, which allows us to show $h_{n+2}h_n=0$ for all $n$: for $n=1$ this follows from $h_1=0$, and for $n>1$, $E(n+1,n,1,0)$ allows us to deduce $h_{n+2}h_n=0$ from 
$h_{n+1}h_{n-1}=0$. Therefore, nonzero terms cannot be exactly two indices apart from each other. $E(n+1,n-1,1,0)$ or more generally $E(m,n,1,0)$ does not add any constraint since it reduces to $h_{m+1}h_{m-1}h_n^2=h_{n+1}h_{n-1}h_m^2$ which is always true since both sides are zero. Imposing all $E(m,n,r,0)$ does introduce additional constraints and implies that the sequence is elliptic, and a complete classification will be stated in Theorem \ref{classification_thm} with a proof preceding it.

In case (101), from the even-odd recurrence we can prove by induction that $h_{2n}=0$ and $h_{2n+1}\ne0$ for all $n$: the $n\le 1$ cases are clear, so assume $n\ge 2$. Then $E(2n,2n-2,1,0)$ says $h_{2n-1}(h_{2n}^2h_{2n-3}-h_{2n+1}h_{2n-2}^2)=0$. Since $h_{2n-2}=0$ and $h_{2n-1}h_{2n-3}\ne0$ by induction hypothesis, we conclude that $h_{2n}=0$. Now in the right-hand side of $E(n+1,n,1,0):h_{2n+1}h_1^2=h_{n+2}h_n^3-h_{n+1}^3h_{n-1}$, every $h_k$ has $k<2n+1$, so by induction hypothesis exactly one term is nonzero depending on the parity of $n$, so $h_{2n+1}\ne0$. In fact we can prove by induction
\[ \frac{h_{2n+1}}{h_1} = (-1)^{n(n-1)/2}\left(\frac{h_3}{h_1}\right)^{n(n+1)/2},
\footnote{Compare Theorem 4.4.3 in \cite{Swart}, Theorem 23.1 in \cite{Ward}.}\]
and we can verify that $h/h_1$ agrees with
$\EDS(0,h_3/h_1,D)=\EDS(0,h_3/h_1,0)$, where $D$ is arbitrary, so $h$ in fact satisfies all elliptic relations.
The Somos 4 recurrence $E(n,2,1,0)$ reduces to $h_{n+2}h_{n-2}h_1=-h_n^2h_3$ in this case, which by induction again shows that $h=h_1\EDS(0,h_3/h_1,0)$ (for the even terms $h_{n-2}=0$ implies $h_n=0$, while for the odd terms $h_{n+2}$ can be computed from $h_n$ and $h_{n-2}$). Therefore either of the even-odd recurrence and the Somos 4 recurrence implies all $E(a,b,c,d)$ in this case.

In case (100), the Somos 4 recurrence simply says there are no nonzero terms exactly four indices apart. The even-odd recurrence is more restrictive, but still allows an infinite number of arbitrary terms: we may specify the even terms arbitrarily as long as we keep the nonzero terms at least four indices apart: from the ``odd recurrence" $E(n+1,n,1,0)$ we recursively compute that all odd terms except $h_1$ are zero, since each term in the right-hand side has two factors that are only two indices apart. Then the ``even recurrence" $E(n+1,n-1,1,0)$ is trivially satisfied since each term contains an odd term factor.\footnote{It is still an interesting problem to determine all constraints from the even-odd recurrence: by considering $E(n+1,n-1,1,0)$ for $8\le n\le 20$ we can derive constraints like $h_4h_6h_8=h_4h_6h_{12}=h_6h_8h_{12}=h_6h_8h_{10}h_{14}=h_6h_8h_{16}=h_8h_{10}h_{16}=h_6h_8h_{10}h_{18}=h_4h_6h_{16}=h_4h_6h_{18}=h_8h_{10}h_{12}h_{18}=h_8h_{10}h_{20}=h_{10}h_{12}h_{18}h_{20}=h_4h_6h_{20}h_{22}=0$, but further constraints might require more than one term.}
However, as soon as we impose all $E(m,n,1,0)$ (equivalently, all $E(a,b,c,d)$, by Corollary \ref{ell_rel_from_one_zero}, since $h_1\ne0$), the possibilities narrow down to (infinitely many) 3-parameter families\footnote{Compare Theorem 4.4.5 in \cite{Swart}, Theorem 25.1 in \cite{Ward}.}: for some odd $\ell>3$, 
  \begin{equation} \label{ES1}
    h_n=
    \begin{cases}
      -h_1 B^{k(k+1)/2}D^{k(k-1)/2}, & \text{if }n=k\ell-1 \\
      +h_1 B^{k(k-1)/2}D^{k(k+1)/2}, & \text{if }n=k\ell+1 \\
      0, & \text{if }n\not\equiv\pm1\pmod\ell
    \end{cases}
  \end{equation}
  where $B:=-h_{\ell-1}/h_1$ and $D:=h_{\ell+1}/h_1$.
These sequences are not multiples of standard EDSs, though if we take $\ell=3$ in case (iii) then it agrees with $h_1\EDS(h_{\ell-1}/h_1,0,h_{\ell+1}/h_{\ell-1})$, and the $\ell>3$ case is obtained by periodically inserting zero terms into this sequence. If both of $h_{\ell\pm1}$ are zero, then $h_1$ is the only nonzero term and the sequence agrees with $h_1\EDS(0,0,0)$, and if only one of them is zero, then the other together with $h_1$ are the only nonzero term. An $\Z^+$-indexed sequence like this with at most one nonzero even term and at most one nonzero odd term is always elliptic, since every term in every elliptic relation $E(a,b,c,d)$ with integers $a>b>c>d\ge0$ involves two factors with different indices of the same parity ($a\pm b$, $a\pm c$, or $b\pm c$).

To show that $h$ must be of the form (\ref{ES1}), consider the first nonzero term $h_{\ell-1}$ after $h_1$: if no such term exists we can take $B=D=0$ and $\ell$ arbitrary, and if one exists the index $\ell-1$ must be even: suppose not and there is no nonzero term between $h_1\ne0$ and $h_{2n+1}\ne0$, then $n\ge2$ since we assumed $h_3=0$, so $E(n+1,n,1,0)$ implies that $h_{n+2}h_n\ne0$ or $h_{n-1}h_{n+1}\ne0$, but both $n$ and $n+1$ are strictly between 1 and $2n+1$, contradiction. Since $h_{\ell-1}h_1\ne0$, by the identity principle Proposition \ref{even_odd_ext} with $r=1$ and $s=\ell-1$, we only need to show the identity (\ref{ES1}) for $n=1,3$ and even $n\le\ell+1$ (which is clear), and that any sequence defined by (\ref{ES1}) is elliptic (which is the $r=1$, $s=\ell-1$ case of the following theorem).
\begin{Theorem}
    Let $r,s\in\Z^+$ be of opposite parity and $r<s$, and let $A,B,D$ be elements in a commutative ring $R$. Then the sequence $(h_n)_{n\in\Z}$ defined by 
\begin{equation}
h_m=
    \begin{cases}
      -AB^{k(k+1)/2}D^{k(k-1)/2} & \text{if }m=k(r+s)-r \\
      +AB^{k(k-1)/2}D^{k(k+1)/2} & \text{if }m=k(r+s)+r \\
      0 & \text{if }m\not\equiv\pm r\pmod{r+s}
    \end{cases}
\end{equation}
which we denote by $A\cdot\ES_{r,s}(B,D)$, is elliptic.\footnote{As noted earlier, the $r=1$ case of these elliptic sequences were known to Ward, but the general case is not found in the literature, even in \cite{Swart}, where generalised elliptic sequences without the $h_1=\pm1$ requirement are considered.}
\end{Theorem}
\begin{proof}
     This sequence clearly satisfies $h_0=0$ and $h_{-n}=-h_n$, so we can arbitrarily permute $a,b,c,d$ without affecting the truth of $E(a,b,c,d)$. If all three terms in $E(a,b,c,d)$ are zero then there is nothing to prove, so let us assume without loss of generality that the second term is not zero, so $a\pm c, b\pm d\equiv \pm r\pmod{r+s}$. Writing $a+c=k_+(r+s)\pm r$ and $a-c=k_-(r+s)\pm r$, we get the solutions $a=\frac{k_+ + k_-}2 (r+s)\pm r$ and $c=\frac{k_+ -k_-}2 (r+s)$, or $a=\frac{k_+ + k_-}2 (r+s)$ and $c=\frac{k_+ -k_-}2 (r+s)\pm r$. Since $r+s$ is odd and $a$ and $c$ are integers, $k_+$ and $k_-$ must have the same parity. Writing $k_a=\frac{k_+ + k_-}2$ and $k_c=\frac{k_+ -k_-}2$ and swapping $a$ and $c$ if necessary, we have $a=k_a(r+s)+\sigma_a r$, $c=k_c(r+s)$ and similarly $b=k_b(r+s)+\sigma_b r$ and $d=k_d(r+s)$, where $\sigma_a,\sigma_b=\pm1$. $E(a,b,c,d)$ then becomes
\begin{align*}
    & h_{(k_a+k_b)(r+s)+(\sigma_a+\sigma_b) r}h_{(k_a-k_b)(r+s)+(\sigma_a-\sigma_b) r}h_{(k_c+k_d)(r+s)}h_{(k_c-k_d)(r+s)} \\
    =& h_{(k_a+k_c)(r+s)+\sigma_a r}h_{(k_a-k_c)(r+s)+\sigma_a r}h_{(k_b+k_d)(r+s)+\sigma_b r}h_{(k_b-k_d)(r+s)+\sigma_b r} \\
    - & h_{(k_a+k_d)(r+s)+\sigma_a r}h_{(k_a-k_d)(r+s)+\sigma_a r}h_{(k_b+k_c)(r+s)+\sigma_b r}h_{(k_b-k_c)(r+s)+\sigma_b r}
\end{align*}
The first term is always zero because $h_m=0$ for $m\equiv 0\pmod{r+s}$. For the four choices of $(\sigma_a,\sigma_b)$, we can verify that both terms on the right-hand side always evaluate to $A^4 B^{K-k}D^{K+k}$, where $K:=k_a^2+k_b^2+k_c^2+k_d^2$ and $k:=\sigma_a k_a+\sigma_b k_b$, thereby verifying the elliptic relations.
\end{proof}

Before we return to case (0), we first study dilation and contraction of elliptic sequences. If $\ell\in\Z^+$ and $(h_n)_{n>0}$ satisfies $E(\ell a,\ell b,\ell c,\ell d)$, then the \textbf{$\ell$-contraction} $h^{[\ell]}$ of $h$, defined by $h^{[\ell]}_n:=h_{n\ell}$, satisfies $E(a,b,c,d)$. In particular, if $h$ is elliptic then so is $h^{[\ell]}$, and if $h$ satisfies all $E(m,n,r,0)$ then so does $h^{[\ell]}$. If $\ell$ is odd and $(h_n)_{n>0}$ is elliptic, then so does the \textbf{$\ell$-dilation} $h^{[1/\ell]}$ of $h$, which is defined by $h^{[1/\ell]}_{n\ell}:=h_n$ and $h^{[1/\ell]}_n:=0$ if $\ell\nmid n$. This is because if any of the three terms in $E(a,b,c,d)$ is nonzero when applied to $h^{[1/\ell]}$, say the first term, then $\ell\mid a\pm b$ and $\ell\mid c\pm d$, which implies that $\ell$ divides $2a,2b,2c,2d$ and therefore divides $a,b,c,d$ since $\ell$ is odd, so $E(a,b,c,d)$ for $h^{[1/\ell]}$ follows from $E(a/\ell,b/\ell,c/\ell,d/\ell)$ for $h$. If all three terms in $E(a,b,c,d)$ are zero then of course it holds.

Let us now focus on sequences of the form $h=H^{[1/2]}$, in which all odd terms vanish. In order for such a sequence to satisfy all $E(m,n,r,0)$, firstly the even subsequence $h^{[2]}=H$ need to satisfy all $E(m,n,r,0)$ (since $h$ satisfies $E(2m,2n,2r,0)$),
and secondly either all even terms of $h^{[2]}$ vanish or all odd terms vanish: $E(m,n,r,0)$ with e.g. $m$ even and $n,r$ odd gives rise to $h_{n+r}h_{n-r}h_m^2=0$, so if there is at least one nonzero even term $h_m$, then two nonzero terms of $h$ cannot be twice an odd number ($2r$) of indices apart, so two nonzero terms of $h^{[2]}$ cannot be an odd number of indices apart. Iterating this argument, if $h$ is not identically zero, it must be of the form $H^{[1/2^k]}$ for some $k$ and some sequence $(H_n)_{n>0}$ satisfying $E(m,n,r,0)$ all whose nonzero terms have odd indices.

Conversely, if $H$ satisfies all $E(a,b,c,d)$ and either all its even terms or all its odd terms are zero, we show that $h := H^{[1/2]}$ also satisfies all $E(a,b,c,d)$. If $a,b,c,d\in\Z$ are all even or all odd, then every term in $E(a,b,c,d)$ is the product of four even term factors, and $E(a,b,c,d)$ for $h$ is equivalent to $E(a/2,b/2,c/2,d/2)$ for $H$. If the number of even parameters is 1 or 3, then every term of $E(a,b,c,d)$ has two even term factors and two odd term factors and therefore vanishes. If the number of even parameters is 2, then one of the terms has four even term factors and two of them are $2\cdot\text{odd}$ apart (and therefore their product vanishes), while the other two terms has four odd term factors (which vanish).

Now we are ready to return to case (0) to classify sequences satisfying all $E(m,n,r,0)$.
Sequences with at most one odd term and at most one even term have been shown to be elliptic, so assume there are at least two terms in the same parity class, and call a pair of distinct positive integers $(k,l)$ of the same parity a \textbf{nonzero pair} if $h_kh_l\ne0$.
Let $(r,r+2n)$ be the lexicographically smallest nonzero pair. If $r > n$, then $E(r+n,r,n,0):h_{n+2r}h_n^3=h_{r+2n}h_r^3-h_{r+n}^3h_{r-n}$ implies that at least one of $(n,n+2r)$ and $(r-n,r+n)$ is a smaller nonzero pair, so we must have $r\le n$. 

If $r=n$, first suppose that all nonzero terms are in the same parity class (as $r$), and that $r$ is odd, so all even terms are zero (if $r$ is even, we can write $r=2^k r'$ with $r'$ odd, so by discussions above $h=H^{[1/2^k]}$ for some $H$ satisfying all $E(m,n,r,0)$ with all even terms zero, and we consider $H$ instead). 
Then the recurrence $E(2r+i,r+i,r,0):h_{3r+2i}h_r^3=h_{3r+i}h_{r+i}^3-h_{2r+i}^3h_i$ uniquely determines $h_{3r+2i}$ in terms of previous terms. 
Since all terms up to $h_{3r}$ are known (there is no nonzero odd term except $h_r$ by minimality of $(r,r+2n)=(r,3r)$), by induction $h$ coincides with $h_r\EDS(0,h_{3r}/h_r,0)^{[1/r]}$.

Next assume that $r=n$ and there exists some $s$ of parity opposite to $r$ with $h_s\ne0$, and we choose $s$ to be minimal; we show that it must be the case that $s=2r$, and therefore $r$ is odd, $h_{2r}$ is the first nonzero even term, and $h_r$ is the first nonzero term: suppose $s\ne 2r$, then $E(2r,r,s,0)$ (with the 
parameters reordered to be antitone) implies that either $(\abs{2r-s},2r+s)$ is a nonzero pair, or $(\abs{r-s},r+s)$ is a nonzero pair and 
$h_{2r}\ne0$. But the former case cannot happen, because $s\le \abs{2r-s}$ by minimality of $s$, which implies that $s<r$ and therefore $(s,2r-s)$ is a smaller nonzero pair, contradicting the choice of $(r,r+2n)$. So $(\abs{r-s},r+s)$ is a nonzero pair and we must have $r\le\abs{r-s}$, which implies that $2r\le s$. Moreover, $h_{2r}\ne0$, so if $r$ is even, then $(r,2r)$ would be a smaller nonzero pair than $(r,3r)$, while if $r$ is odd then $s=2r$ by minimality of $s$, leading to a contradiction in both cases.

In general, if $h_k$ is the first nonzero term of a sequence satisfying all $E(m,n,r,0)$, then every nonzero pair $(c,d)$ satisfies $d\ge c+2k$; we say the gap of a nonzero pair must be at least $2k$. This is because if $i < k$, then we can show by induction that $h_{c+2i}h_c=0$ for all $c$: if $c \le i$ then $c < k$ so $h_c=0$, and for $c > i$ we use $E(c+i,c,i,0): h_{2c+i}h_i^3=h_{c+2i}h_c^3-h_{c+i}^3h_{c-i}$: the first term is zero since $h_i=0$, and $h_{c+i}h_{c-i}=0$ by induction hypothesis, so we conclude that $h_{c+2i}h_c=0$. In the current situation we have that $h_r,h_{2r}$ and $h_{3r}$ are the only nonzero terms up to $3r$, so we can conclude they are also the only nonzero terms below $2r+2r=4r$. Applying the identity principle Proposition \ref{even_odd_ext} with $s=2r$ we conclude that $h=h_r\EDS(h_{2r}/h_r,h_{3r}/h_r,h_{4r}/h_{2r})^{[1/r]}$.

Having dealt with the $r=n$ case, we now turn to the $r<n$ case. $E(r+n,n,r,0)$ implies that at least one of $(n-r,n+r)$ and $(n,n+2r)$ is a nonzero pair, but $(n-r,n+r)$ cannot be a nonzero pair, so $(n,n+2r)$ must be. To show this we study two cases separately: if $n-r\le r$, since $n+r<r+2n$ always holds, $(n-r,n+r)$ cannot be a nonzero pair by minimality of $(r,r+2n)$. If $n-r>r$, assume for contradiction that $(n-r,n+r)$ is a nonzero pair. Then $n-r$ must have parity opposite to $r$, since otherwise $(r,n-r)$ would be a smaller nonzero pair. There cannot be any nonzero term $h_s$ below $r$ in the parity class of $n-r$, since otherwise $(s,n-r)$ would be a smaller nonzero pair; therefore $h_r$ is the first nonzero term. By induction we can now show that $h_{k+2r}h_k=0$ for $k<n$: for $k<r$ this is true since $h_k=0$; for $k=r$ this is true since $h_{3r}=0$ (since $3r<r+2n$); for $k>r$ we use $E(k+r,k,r,0):h_{2k+r}h_r^3=h_{k+2r}h_k^3-h_{k+r}^3h_{k-r}$: we have $h_{2k+r}=0$ because $2k+r<r+2n$ and $h_{k+r}h_{k-r}=0$ by induction hypothesis, from which we conclude $h_{k+2r}h_k=0$. Taking $k=n-r$ we see that $h_{n+r}h_{n-r}=0$, contradicting the assumption that $(n-r,n+r)$ is a nonzero pair.

Now that we know that $(n,n+2r)$ is a nonzero pair, $n$ must have parity opposite to $r$, since otherwise $(r,n)$ would be a smaller nonzero pair.
Moreover, $h_n$ is the first nonzero term in the parity class of $n$ (opposite to the parity of $r$): suppose not, and let $s<n$ be the first nonzero term in the parity class. Then $r<s$ (otherwise $(s,n)$ is a smaller nonzero pair than $(r,r+2n)$), so $h_r$ is the first nonzero term of $h$, and the gap of a nonzero pair must be at least $2r$, so $s+2r\le n$, and the only nonzero term in the parity class of $s$ below $s+2r$ is $h_s$. Of course the only nonzero term in the parity class of $r$ up to $3r<r+2n$ is $h_r$ by minimality of $(r,r+2n)$, so by the identity principle Proposition \ref{even_odd_ext} again, we conclude that $h=h_r\ES_{r,s}(-h_s/h_r,h_{s+2r}/h_r)$. If $h_{s+2r}=0$, the only nonzero terms of $h$ are $h_r$ and $h_s$, contradicting $h_n\ne0$, while if $h_{s+2r}\ne0$, then $h_{r+2s}\ne0$, contradicting minimality of $(r,r+2n)$.

Now we are reduced to the situation that $r<n$ are of opposite parity, $h_r$ and $h_{r+2n}$ are the first two nonzero terms in the parity class of $r$, and $h_n$ is the first nonzero term in its parity class. Since the gap of a nonzero pair is at least $2r$, $h_n$ is the only nonzero term below $n+2r$ in its parity class. Applying Proposition \ref{even_odd_ext} again, we conclude that $h=h_r\ES_{r,n}(-h_n/h_r,h_{n+2r}/h_r)$.

In summary, we have
\begin{Theorem} \label{classification_thm}
    Let $(h_n)_{n>0}$ be a sequence over a field $R$ satisfying all $E(m,n,r,0)$. Then $h$ falls in one of the following cases:
    \begin{enumerate}[(i)]
    \item there exist an odd $r\in\Z^+$ and $A,B,C,D\in R$ such that $h=A\cdot\EDS(B,C,D)^{[1/r]}$,
    \item there exist $r,s\in\Z^+$ of opposite parity with $r<s\ne 2r$ and $A,B,D\in R$ such that $h=A\cdot\ES_{r,s}(B,D)$,
    \item there exist an even $r\in\Z^+$ and $A,C\in R$ such that $h=A\cdot\EDS(0,C,0)^{[1/r]}$,
    \end{enumerate}
    and all of these sequences are elliptic.
\end{Theorem}
By Theorem \ref{neg_ext}, $\Z$-indexed sequences satisfying all $E(m,n,r,0)$ also must be one of these three forms.
We require $s\ne 2r$ in (ii) and $r$ be even in (iii) in order for the cases not to be subsumed by (i). For each odd $r$, the closure of the set of type (i) sequences form a 4-dimensional rational irreducible component in the algebraic set of elliptic sequences over the field. We have not determined whether the closure could contain all type (ii) or type (iii) sequences of some suitable parameter, so we are not sure whether type (ii) and type (iii) sequences also form irreducible components. 
The coordinate ring of the whole algebraic set is not finite type, since there are infinitely many connected components. Even though finite-dimensional, it looks plausible that the coordinate rings of the irreducible components are also not of finite type. For type (ii) and (iii) sequences, the coordinate rings are isomorphic to $R[AB^{k(k+1)/2}D^{k(k-1)/2}]_{k\in\Z}\subset R[A,B,D]$ and $R[AC^{n(n+1)/2}]_{n\in\N}\subset R[A,C]$ respectively, which can be shown not to be of finite type (the former is not finite type even after specializing to $A=1$).

Finally, we analyze type (11) sequences satisfying the Somos 4 recurrence $E(n,2,1,0)$.

\subsection{Sequences $(h_n)_{n>0}$ satisfying $E(n,2,1,0)$ with $h_2 h_1\ne0$ and $h_3=0$} Let us call these sequences type (110) Somos sequences. $E(n,2,1,0)$ says $h_{n+2}h_{n-2}h_1^2=h_{n+1}h_{n-1}h_2^2$ in this case: in other words, whenever $(n+1,n-1)$ is a nonzero pair, $(n+2,n-2)$ must also be. If we use a binary string $ABCD$ to represent whether each of $h_{n-2},h_{n-1},h_n,h_{n+1}$ is zero (e.g. $A=0$ means that $h_{n-2}=0$ and $A=1$ means that $h_{n-2}\ne0$), then it can be extended to $ABCDE$ if and only if $A=1$ (in which case $E=B\cdot D$) or $A=0$ and $B\cdot D=0$ (in which case $E=0$ and $E=1$ are both valid extensions, and $E=1$ leads to a new free parameter), and if $ABCD$ extends to $ABCDE$ we say that $ABCD\to BCDE$ is a valid transition. By assumption the initial state is either 1100 or 1101. We can draw the following diagram of valid transitions:
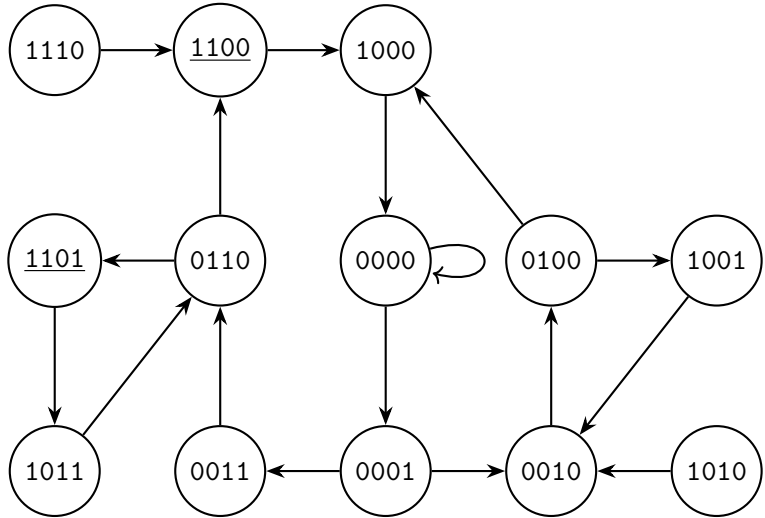
\begin{figure}[H]
    \centering
\begin{tikzpicture}[
  every node/.style={
    draw, circle, minimum size=1.1cm, font=\ttfamily\small
  },
  every path/.style={
    -{Stealth[length=7pt]}, thick
  }
]

\node (1110) at (0,    0)    {1110};
\node (1101) at (0,   -2.8)  {\underline{1101}};
\node (1011) at (0,   -5.6)  {1011};
 
\node (1100) at (2.2,  0)    {\underline{1100}};
\node (0110) at (2.2, -2.8)  {0110};
\node (0011) at (2.2, -5.6)  {0011};
 
\node (1000) at (4.4,  0)    {1000};
\node (0000) at (4.4, -2.8)  {0000};
\node (0001) at (4.4, -5.6)  {0001};
 
\node (0100) at (6.6, -2.8)  {0100};
\node (0010) at (6.6, -5.6)  {0010};
 
\node (1001) at (8.8, -2.8)  {1001};
\node (1010) at (8.8, -5.6)  {1010};
 
 
\draw (1110) -- (1100);
 
\draw (1101) -- (1011);
\draw (1011) -- (0110);
\draw (0110) -- (1101);
 
\draw (0011) -- (0110);
\draw (0110) -- (1100);
\draw (1100) -- (1000);
\draw (1000) -- (0000);
\draw (0000) -- (0001);
\draw (0001) -- (0011);
 
\draw (0001) -- (0010);
\draw (0010) -- (0100);
\draw (0100) -- (1000);
 
\draw (0100) -- (1001);
\draw (1001) -- (0010);
 
\draw (1010) -- (0010);
 
\draw (0000) edge[loop right] (0000);
 
\end{tikzpicture}
\end{figure}

Notice that there are five loops in the diagram: 0, 001 (i.e. $0010\to 0100\to 1001\to 0010$), 011, 00001, and 000011. Repeating these loops produces families of sequences with infinitely many parameters: 001, 011 and 00001 produces one free parameter per period, while 000011 produces two. An infinite binary string whose length 5 substrings all represent valid transitions should correspond to an irreducible component in the space of type (11) Somos sequences, if changing any nonzero number of 0 in the string to 1 produces an invalid string. Such a component has infinite dimension whenever the string contains infinitely many 1.

The two nodes 1010 and 1110 without incoming arrows are not reachable; 1111 is also unreachable with a loop to itself and not drawn. 0101 and 0111 are dead ends (patterns that do not appear in any type (11) Somos sequence) and not drawn on the diagram; if they are drawn then there should be an arrow from 0010 to 0101 and an arrow from 0011 to 0111. 

\subsection{Sequences $(h_n)_{n>0}$ satisfying $E(n,2,1,0)$ with $h_3 h_2 h_1\ne0$} Let use call these strings type (111) Somos sequences. $E(n,2,1,0)$ says $h_{n+2}h_{n-2}=(h_2/h_1)^2 h_{n+1}h_{n-1}-(h_3/h_1)h_n^2$ in this case.
If $h':=h_1\EDS(h_2/h_1,h_3/h_1,h_4/h_2)$ consists purely of nonzero terms then the Somos 4 recurrence uniquely determines $h$ and $h=h'$, but $h'$ may well have zero terms. If this happens, let $h'_n$ be the first zero term, so $n\ge 4$. We still have $h_m=h'_m$ for $m\le n+3$, but $h_{n+4}$ can be arbitrary. The condition $h'_n=0$ (and $h'_m\ne0$ for $0<m<n$) defines a hypersurface,
over which the main irreducible component (4-dimensional, parameterized by $h_1,h_2,h_3,h_4$) ``branches off" another component (or multiple components), which is still 4-dimensional with the introduction of the new free variable $h_{n+4}$.

An intriguing observation is that this ``branching off" process is self-similar in a sense: equations for the branch-off hypersurfaces of the main component and of a branched off component are related by a simple transformation, which we now derive. Firstly, using $h_{n-1}\ne0$, $h_n=0$ and $E(m,2,1,0)$ for $m=n-1,n,n+1$ we iteratively prove that $h_{n+1},h_{n+2},h_{n+3}$ are all nonzero. We can then write $h_{n+1}=r_1 h_1$ and $h_{n+2}=r_1 r_2 h_2$, and then $E(n+2,2,1,0)$ gives rise to the identity $h_{n+3}=r_1 r_2^2 h_3$. If we define a sequence $(h^{[n+]}_m)_{m>0}$ by $h^{[n+]}_m=h_{n+m}/r_1 r_2^{m-1}$, then the first four terms of $h^{[n+]}$ agrees with $h_1\EDS(h_2/h_1,h_3/h_1,h^{[n+]}_4/h_2)$, and both sequences satisfy all $E(m,2,1,0)$ ($h^{[n+]}$ satisfies them because $h$ satisfies $E(n+m,2,1,0)$). Therefore, $h_{n+m}=r_1 r_2^{m-1} h_1\EDS(h_2/h_1,h_3/h_1,h_{n+4}/r_1 r_2^3 h_2)$ as long as $h_{n+k}\ne0$ for $0<k\le m-4$. We therefore see that if the $m$th branch-off hypersurface is defined by the equation $F(h_1,h_2,h_3,h_4)=0$, then the $m$th branch-off hypersurface of a branched off component at index $n$ is defined by $F(h_1,h_2,h_3,h_{n+4}/r_1 r_2^3)=0$ (notice that $r_1$ and $r_2$ are rational functions in $h_1,h_2,h_3,h_4$ and therefore algebraic functions in $h_1,h_2,h_3$ on the branched off component). A very similar phenomenon has been observed by Ward (the ``symmetry formula"), see \cite{Swart}, Theorem 4.7.6. 

Branching off at index $n$ is only possible if $h'_n=0$ does not imply some $h'_m$ with $m<n$ must be zero. This is in other words saying that $h'_n$ (as a polynomial in $h_1,h_2/h_1,h_3/h_1,h_4/h_2$) has a primitive divisor, or that the Zsigmondy set of the universal EDS over the field 
does not contain $n$. We are not sure whether this Zsigmondy set is always empty, or whether it is the same for every field. See \cite{IS, EMR} for results on Zsigmondy set of EDSs.

\subsection{The Scheme of Elliptic Sequences} Instead of fixing a field, we may also consider the scheme of ($\Z^+$-indexed) elliptic sequences, defined as the prime spectrum of the universal ring $\Z[h_n]_{n>0}/\langle E(a,b,c,d) \rangle_{a>b>c>d\ge0}$, where each $h_n$ is now a indeterminate in the polynomial ring. Elliptic sequences over a commutative ring $R$ correspond to ring homomorphisms from the universal ring to $R$, which in turn correspond to $R$-points of the scheme. 
Our study of field-valued points determines the reduced induced structure of the scheme: type (i) sequences give rise to 5-dimensional irreducible components in the scheme, and type (ii) and (iii) sequences might give rise to lower-dimensional ones. 

An interesting question is whether the scheme (or the universal ring) is reduced, or equivalently, whether the ideal $\langle E(a,b,c,d) \rangle_{a>b>c>d\ge0}$ is radical. It is not quite amenable to computational approaches due to infinitely many indeterminates and relations: Macaulay2 readily shows that $\Z[h_n]_{0<n\le 6}/\langle E(3,2,1,0),E(4,2,1,0) \rangle$ is not reduced, but the explicit square-zero element it gives becomes zero once we add in $h_7$ and associated relations, as it equals $h_2 h_3 E(4,3,1,0)-h_1 h_2 E(5,2,1,0)$\footnote{Note that $E(5,2,1,0)=E(4,3,2,1)$.}. If it is not reduced, it would be interesting to determine whether each irreducible component is also not reduced. If it is reduced, we can still ask whether $\Z[h_n]_{n>0}/\langle E(m,n,r,0)\rangle_{m>n>r>0}$ is reduced: our classification of elliptic sequences show that the collection of all $E(m,n,r,0)$ implies all $E(a,b,c,d)$ over fields, so the radical of the ideals $\langle E(m,n,r,0)\rangle_{m>n>r>0}$ and $\langle E(a,b,c,d)\rangle_{a>b>c>d\ge0}$ are the same. Macaulay2 shows that $E(5,3,2,1)$ (the first elliptic relation whose terms consist of four distinct factors) is nonzero in $\Z[h_n]_{0<n\le 8}/\langle E(m,n,r,0)\rangle_{m>n>r>0, m+n\le8}$ while its cube is zero, but once we add $h_9, h_{10}$ and associated relations, then its square is already zero; adding more terms/relations might make itself zero, but $h_{11}$ and $h_{12}$ are not enough. 
Results about elliptic sequences modulo prime powers in e.g. \S4.7 and \S5 
of \cite{Swart} could potentially shed light on these questions.

If we consider the universal ring $\Z[h_n]_{n\in\Z}/\langle E(a,b,c,d)\rangle_{a,b,c,d\in\Z}$ for $\Z$-indexed elliptic sequences, it is indeed not reduced, and $h_0$ is a nonzero nilpotent element: we have $h_0^4=0$ because of $E(0,0,0,0)$, but $h_0\ne0$ in the universal ring, because there exists an elliptic sequence with $h_0\ne0$: take any $R$ with an element $r\ne0$ with $r^4=0$ and define $h_n=r$ for all $n\in\Z$.

\section{Acknowledgements} 

We thank David Kurniadi Angdinata for bringing up the division polynomial problem that motivated this work, and for suggestions to improve the paper.
We thank Kevin Buzzard, Antoine Chambert-Loir, and Jinzhao Pan for helpful discussions on the Lean Zulip server.

\end{document}